\documentclass[10pt]{article}
\usepackage[ansinew]{inputenc}
\usepackage{amsmath}
\usepackage{amsfonts}
\usepackage{amssymb}
\usepackage{amsthm}
\usepackage{newlfont}
\usepackage{amsbsy}
\usepackage{bm}
\usepackage{color}
\usepackage{epsfig}
\usepackage{dsfont}
\usepackage{latexsym}
\usepackage{array}
\usepackage{longtable}
\usepackage{enumerate}
\usepackage{multicol}
\usepackage{mathrsfs}
\usepackage{wasysym}
\usepackage{graphics}
\usepackage{graphicx}
\usepackage{hyperref}
\hypersetup{
    colorlinks,
    citecolor=blue,
    filecolor=black,
    linkcolor=red,
    urlcolor=black
}

\newtheorem{lema}{Lemma}[section]
\newtheorem{teorem}{Theorem}[section]

\newtheorem{defin}{Definition}[section]
\newtheorem{prop}{Proposition}[section]
\newtheorem{rem}{\it Remark}[section]
\newcommand{\nor}[2]{{\left\|{#1}\right\|_{#2}}}
\newcommand{\nora}[3]{{\left\|{#1}\right\|_{#2}^{#3}}}

\title{\textbf{On the Well-posedness for the Chen-Lee equation in periodic Sobolev spaces}}
\author{Ricardo A. Pastr\'an R. \thanks{Ricardo Pastr\'an. Departamento de Matem\'aticas, Universidad Nacional de Colombia, Bogot\'a, Colombia. Carrera 30 No. 45-03. C\'odigo postal 11321.
E-mail: {\tt rapastranr@unal.edu.co}}\\ Oscar G. Ria\~no C. \thanks{Departamento de Matem\'aticas, Universidad Nacional de Colombia, Bogot\'a, Colombia. Carrera 30 No. 45-03. C\'odigo postal 11321.
E-mail: {\tt ogrianoc@unal.edu.co}}}

\begin{document}

\maketitle

\begin{abstract}
We prove that the initial value problem associated to a perturbation of the Benjamin-Ono equation or Chen-Lee equation $u_t+uu_x+\beta \mathcal{H}u_{xx}+\eta (\mathcal{H}u_x - u_{xx})=0$, where $x\in \mathbb{T}$, $t> 0$, $\eta >0$ and $\mathcal{H}$ denotes the usual Hilbert transform, is locally and globally well-posed in the Sobolev spaces $H^s(\mathbb{T})$ for any $s>-\frac{1}{2}$. We also prove some ill-posedness issues when $s<-1$. 
\end{abstract}


\section{Introducción}\label{section}

The goal in this paper is to establish well-posedness results on the Cauchy pro\-blem
\begin{equation}\label{cl}
CL \left\{
\begin{aligned}
u_t+uu_x+\beta \mathcal{H}u_{xx} + \eta (\mathcal{H}u_x-u_{xx})&=0 \qquad t>0, \quad  x\in \mathbb{T},\\
u(x,0)&=\phi(x),
\end{aligned}
\right.
\end{equation}
where $\beta,\,\eta >0$ are constants. In the equation, $\mathcal{H}$ denotes the usual Hilbert transform given by
\begin{equation*}
\mathcal{H}f(x)=\frac{1}{2\pi}\text{p.v.}\int_{-\pi}^{\pi} \cot\Bigl(\frac{x-y}{2}\Bigr)f(y) \ dy=\frac{i}{2}(\text{sgn}(k)\widehat{f}(k))^{\vee}(x);\;\,k \in \mathbb{Z},\, f\in \mathcal{P}.
\end{equation*}
This equation was first introduced by H. H. Chen and Y. C. Lee in \cite{CL} to describe fluid and plasma turbulence. It deserves to remark that the fourth and the fifth terms represent, respectively, instability and dissipation. Several authors have studied this equation from a numerical standpoint. For example, H. H. Chen, Y. C. Lee and S. Qian in \cite{clq, clq1}, and B. -F. Feng and T. Kawahara, in \cite{FeKa}, who investigated the initial value problem as well as stationary solitary and periodic waves of the equation. Also, R. Pastr\'an in \cite{P} proved using the Fourier restriction norm method that the initial value problem $CL$ is locally well-posed in the Sobolev spaces $H^s(\mathbb{R})$ for any $s>-1/2$, globally well-posed in $H^s(\mathbb{R})$ when $s\geq 0$ and that one cannot solve the Cauchy problem by a Picard iterative method implemented on the integral formulation of $CL$ for initial data in the Sobolev space $H^s(\mathbb{R})$, $s<-1$.

We say that the Cauchy problem or initial value problem (\ref{cl}) is \textit{locally well-posed} in $H^s(\mathbb{T})$ if for any $\phi^* \in H^s(\mathbb{T})$ there exists a time $T>0$ and an open ball $B$ in $H^s(\mathbb{T})$ containing $\phi^*$, and a subset $\mathfrak{X}_T$ of $C([0,T]; H^s(\mathbb{T}))$, such that for each $\phi \in B$ there exists a unique solution $u\in \mathfrak{X}_T$ to the integral equation associated to the Cauchy problem and furthermore the map $\phi \mapsto u$ is continuous from $B$ to $\mathfrak{X}_T$. If we can take $T$ arbitrarily large we say that the initial value problem is \textit{globally well-posed}.

We show that the initial value problem $CL$ is locally and globally well-posed in the Sobolev spaces $H^s(\mathbb{T})$ for any $s>-1/2$. Since the dissipation of the Chen-Lee equation is in some sense ``stronger" than the dispersion, we will use the purely dissipative methods of Dix for the KdV-B equation \cite{Dix}, see also Duque \cite{Duque}, Esfahani \cite{Amin} and Pilod \cite{KDV}, which consist in applying a fixed point theorem to the integral equation associated on $CL$ in an adequate $\mathfrak{X}_T$ space (see (\ref{spacexts}) for the exact definition). We also prove that one cannot solve the Cauchy problem by a Picard iterative method implemented on the integral formulation of $CL$ for initial data in the Sobolev space $H^s(\mathbb{T})$, $s<-1$. In particular, the methods introduced by Bourgain \cite{Bourgain} and Kenig, Ponce and Vega \cite{KPV} for the KdV equation cannot be used for $CL$ with initial data in the Sobolev space $H^s(\mathbb{T})$ for $s<-1$. This kind of ill-posedness result is weaker than the loss of uniqueness proved by Dix in the case of Burgers equation. 
\subsection{Definitions and Notations}
Given $a$, $b$ positive numbers, $a\lesssim b$ means that there exists a positive constant $C$ such that $a\leq C b$. And we denote $a\sim b$ when, $a \lesssim  b$ and $b \lesssim a$. We will also denote $a\lesssim_{\lambda} b$ or $b\lesssim_{\lambda} a$, if the constant involved depends on some parameter $\lambda$. Given a Banach space $X$, we denote by $\nor{\cdot}{X}$ the norm in $X$. We will understand $\langle \cdot \rangle = (1+|\cdot|^2)^{1/2}$.

$\mathcal{P}=C^{\infty}(\mathbb{T})$ denotes the space of all infinitely differentiable $2\pi$-periodic functions and $\mathcal{P}'$ will denote the space of periodic distributions, i.e., the topological dual of $\mathcal{P}$. For $f\in \mathcal{P}'$ we denote by $\widehat{f}$ or $\mathcal{F}(f)$ the Fourier transform of $f$, $\Hat{f}=\left(\Hat{f}(k)\right)_{k\in\mathbb{Z}}$, where $\Hat{f}(k)=\frac{1}{2\pi}\int_{-\pi}^{\pi}e^{-ik\cdot z}f(z)\,dz,$ for all integer $z$. We will use the Sobolev spaces $H^s(\mathbb{T})$ equipped with the norm 
$$\nor{\phi}{H^s}=(2\pi)^{\frac{1}{2}} \nor{(1+k^2)^{s/2}\Hat{\phi}(k)}{l^2(\mathbb{Z})}.$$ 
We will denote $\widehat{u}(k,t)$, $k\in\mathbb{Z}$, as the Fourier coefficient of $u(t)$ respect to the variable $x$. Let $U$ be the unitary group in $H^s(\mathbb{T})$, $s\in \mathbb{R}$, generated by the skew-symmetric operator $-\beta \mathcal{H}\partial_x^2 $, which defines the free evolution of the Benjamin-Ono equation, that is,
\begin{equation}\label{unitarygroup}
U(t)=\exp (itq(D_x)), \;\; U(t)f= \Bigl(e^{itq(\xi)}\Hat{f}\Bigr)^{\vee}\; \text{with}\;\, f\in H^s(\mathbb{T}), \, t\in \mathbb{R},
\end{equation}
where $q(D_x)$ is the Fourier multiplier with symbol $q(k)=\beta \,k \,|k|$, for all $k\in \mathbb{Z}$. Since the linear symbol of equation in (\ref{cl}) is $iq(k)+p(k)$, where $p(k)=\eta \,(k^2-|k|)$ for all $k \in \mathbb{Z}$, we also denote by $S(t)=e^{-(\beta \mathcal{H}\partial_x^2 +\eta (\mathcal{H}\partial_x-\partial_x^2))t}$, for all $t\geq 0$, the semigroup in $H^s(\mathbb{T})$ generated by the operator $-(\beta \mathcal{H}\partial_x^2 +\eta (\mathcal{H}\partial_x-\partial_x^2))$, i.e.,
\begin{align}
S(t)f=\Bigl( e^{i\,q(\xi)\,t -p(\xi)\,t}\Hat{f} \Bigr)^{\vee}\quad \text{for}\quad f\in H^s(\mathbb{T}),\;t\geq 0. \label{semigrupos}
\end{align}
We define the next Banach spaces which are inspired by an adaptation made by Esfahani, in \cite{Amin}, of the spaces originally presented by Dix in \cite{Dix}.
\begin{defin} 
Let $0\leq T\leq 1$ and $s < 0$. We consider $X_T^s$ as the class of all the functions $u\in C\left([0,T];H^s(\mathbb{T})\right)$ such that
\begin{equation}\label{spacexts}
\left\|u\right\|_{X_T^s}:=\sup_{t\in(0,T]}\left(\left\|u(t)\right\|_{H^s}+t^{|s|/2}\left\|u(t)\right\|_{L^2}\right) < \infty.
\end{equation}
\end{defin}
\subsection{ Main Results}
We will mainly work on the integral formulation of the $CL$ equation,
\begin{equation}\label{intequation}
u(t)=S(t)\phi - \int_0^tS(t-t')[u(t')u_x(t')]\,dt' ,\quad t\geq 0.
\end{equation}
\begin{teorem}[Local well-posedness]\label{mainresult}
Let $\beta \geq 0$, $\eta >0$ and $s>-1/2$. Then for any $\phi \in H^s(\mathbb{T})$ there exists $T=T(\nor{\phi}{H^s})>0$ and a unique solution $u$ of the integral equation (\ref{intequation}) satisfying
\begin{align*}
&u\in C([0,T],H^s(\mathbb{T}))\cap C((0,T),H^{\infty}(\mathbb{T})).
\end{align*}
Moreover, the flow map $\phi \mapsto u(t)$ is smooth from $H^s(\mathbb{T})$ to $C([0,T],H^s(\mathbb{T}))\cap C((0,T],H^{\infty}(\mathbb{T}))\cap \mathfrak{X}_T^s$.
\end{teorem}
\begin{teorem}[Global well-posedness]\label{globalresult}
 Let $s>-1/2$ and $\phi \in H^s(\mathbb{T})$. Then the supremum of all $T>0$ for which all the assertions of Theorem \ref{mainresult} hold is infinity. 
\end{teorem}
It is known that the Banach's Fixed Point Theorem cannot be applied to the Benjamin-Ono equation \cite{MolSauTzv}. Here, it is proved that there does not exist a $T>0$ such that (\ref{cl}) admits a unique local solution defined on the interval $[0,T]$ and such that the flow-map data-solution $\phi \mapsto u(t)$, $t\in [0,T]$, is $C^2$ differentiable at the origin from $H^s(\mathbb{T})$ to $H^s(\mathbb{T})$. As a consequence, we cannot solve the Cauchy problem for the $CL$ equation by a Picard iterative method implemented on the integral formulation (\ref{intequation}), at least in the Sobolev spaces $H^s(\mathbb{T})$, with $s<-1$.  
\begin{teorem}\label{malpuestodos}
Fix $s<-1$. Then there does not exist a $T>0$ such that (\ref{cl}) admits a unique local solution defined on the interval $[0,T]$ and such that the flow-map data-solution
\begin{equation}
\phi \longmapsto u(t), \qquad t\in [0,T],
\end{equation}
for (\ref{cl}) is $C^2$ differentiable at zero from $H^s(\mathbb{T})$ to $H^s(\mathbb{T})$.
\end{teorem}
A direct corollary of Theorem \ref{malpuestodos} is the next statement.
\begin{teorem}\label{illposed}
The flow map data-solution for the Chen-Lee equation is not $C^2$ from $H^s(\mathbb{T})$ to $H^s(\mathbb{T})$, if $s<-1$.
\end{teorem}
The layout of this paper is organized as follows: Section $2$ presents some linear estimates. Section $3$ is devoted to establishing a bilinear estimate in the space $X_{T}^s$. Theorems \ref{mainresult} and \ref{globalresult} will be proved in Section $4$, and finally, the proof of the Theorem \ref{malpuestodos} will be done in Section $5$.
\setcounter{equation}{0}
\section{ Linear Estimates }
We start giving the following estimates.

\begin{lema}\label{LLE0}
Let $\lambda>0$, $\eta>0$ and $t>0$ be given. Then
$$\left\||tk^2|^{\lambda} e^{\eta\left(|k|-k^2 \right)t}\right\|_{l^{\infty}(\mathbb{Z})}\lesssim_{\lambda}\left(t^{\lambda}+\eta^{-\lambda}\right) e^{\frac{\eta}{8}\left(t+t^{\frac{1}{2}}\sqrt{t+\frac{16\lambda}{\eta}}\right)}.$$
\end{lema}
\begin{proof}
We have the following inequality
$$|tk^2|^{\lambda} e^{\eta\left(|k|-k^2 \right)t}\leq  \sup_{x\in \mathbb{R}} |x|^{2\lambda} e^{\eta\left(|x|t^{1/2}-x^2 \right)}, \qquad \forall k\in \mathbb{Z}.$$
Let $w_t(x)= x^{2\lambda} e^{\eta\left(xt^{1/2}-x^2 \right)}$, for all $x\geq 0$. Note that $w_t(x)$ tends to $0$ as $x\to \infty$, and  
$$w_t'(x_1)=0 \qquad \Longleftrightarrow \qquad  x_{1}=\frac{1}{4}\left(t^{\frac{1}{2}}+\sqrt{t+\frac{16\lambda}{\eta}}\right).$$
Therefore, the maximum of $w_t$ is attained in $x_1$. So, we can obtain
$$w_t(x_1)\lesssim_{\lambda} \left(t^{\lambda}+\eta^{-\lambda}\right) e^{\frac{\eta}{8}\left(t+t^{\frac{1}{2}}\sqrt{t+\frac{16\lambda}{\eta}}\right)}.$$
This inequality completes the proof.
\end{proof}
\begin{lema}\label{LLE1}
Let $\lambda\geq 0$, $\eta>0$ and $t>0$ be given. Then 
$$\left\||k|^{\lambda} e^{\eta\left(|k|-k^2 \right)t}\right\|_{l^2(\mathbb{Z})}\lesssim_{\lambda} \Upsilon_{\eta}^{\lambda}(t),$$
where
$$\Upsilon_{\eta}^{\lambda}(t):=1+\frac{1}{(\eta t)^{\frac{\lambda}{2}}}+\frac{1}{(\eta t)^{\frac{1+2\lambda}{4}}}.$$
\end{lema}
\begin{proof}
From the fact that $\eta\left(|k|-k^2\right)t\leq -\frac{\eta k^2t}{2}$, for all $k\in \mathbb{Z}$ with $|k|\geq 2$, we deduce 
\begin{equation}\label{LE1}
\left\||k|^{\lambda} e^{\eta\left(|k|-k^2 \right)t}\right\|_{l^2(\mathbb{Z})}^2=\sum_{k=-\infty}^{\infty} |k|^{2\lambda} e^{2\eta \left(|k|-k^2\right)t}\lesssim 1+\sum_{k=2}^{\infty} k^{2\lambda} e^{-\eta k^2t}.
\end{equation}
Let $h\left(x\right):=x^{2\lambda} e^{-\eta x^2t}$ for all $x > 0$. We observe that  
\begin{equation}
h'\left(x\right)= 2\left(\lambda-\eta t x^2\right)x^{2\lambda-1}e^{-\eta x^2 t}. \label{LE2}
\end{equation}
Thus, $h(x)$ reaches its maximum value when $x_{\max}=\sqrt{\frac{\lambda}{\eta t}}$. If $x_{\max}\leq 1$, \eqref{LE2} implies that $h$ is a nonincreasing function on the interval $[1,\infty )$. Hence, with the change of variable $u=\eta x^2t$, we obtain
\begin{align}
\sum_{k=2}^{\infty}k^{2\lambda} e^{-\eta k^2t}\leq \int_{1}^{\infty} x^{2\lambda}e^{-\eta x^2 t} dx &= \frac{1}{2}\left(\frac{1}{\eta t}\right)^{\frac{1+2\lambda}{2}}\int_{\eta t}^{\infty} u^{\frac{2\lambda-1}{2}}e^{-u} dx \nonumber \\
 & \leq \frac{1}{2} \left(\frac{1}{\eta t}\right)^{\frac{1+2\lambda}{2}}\Gamma\left(\frac{1+2\lambda}{2}\right). \label{LE3}
\end{align}

On the other hand, if $x_{\max}>1$, we have from \eqref{LE2} that $h(x)$ is a nondecreasing function on the interval $[1,x_{\max})$ and nonincreasing on the interval $\left(x_{\max},\infty \right)$, but this implies arguing as above that  
\begin{align}
\sum_{k=2}^{\infty} k^{2\lambda} e^{-\eta k^2t} &\leq \left(\frac{\lambda}{\eta t}\right)^{\lambda}e^{-\lambda}+\int_{2}^{\infty} x^{2\lambda}e^{-\eta x^2 t}dx \nonumber\\
 & \leq \left(\frac{\lambda}{\eta t}\right)^{\lambda}e^{-s\lambda}+\left(\frac{1}{\eta t}\right)^{\frac{1+2\lambda}{2}}\Gamma\left(\frac{1+2\lambda}{2}\right). \label{LE4}
\end{align}
Combining \eqref{LE3} and \eqref{LE4}, and taking square root of the resulting expression, we can conclude the lemma.
\end{proof}
We have the next linear estimates
\begin{prop} \label{PROP1} 
Let $0< T\leq 1$, $\eta>0$, $s \in \mathbb{R}$ and $\phi\in H^s(\mathbb{T})$. Then 
\begin{equation}\label{LE5}
\sup_{t\in[0,T]}\left\|S(t)\phi\right\|_{H^s}\leq \left\|\phi\right\|_{H^s}.
\end{equation}
Moreover, if $s< 0$,
\begin{equation}\label{LE6}
\sup_{t\in[0,T]}t^{\frac{|s|}{2}}\left\|S(t)\phi\right\|_{L^2}\lesssim_{s}f_{s,\eta}(T)\left\|\phi\right\|_{H^s},
\end{equation}
where $$f_{s,\eta}(t)=1+ \left(t^{\frac{|s|}{2}}+\eta^{-\frac{|s|}{2}}\right) e^{\frac{\eta}{8}\left(t+t^{\frac{1}{2}}\sqrt{t+\frac{8|s|}{\eta}}\right)},$$
is a nondecreasing function on $[0,1]$.
\end{prop}

\begin{proof}
Since $\eta \left(|k|-k^2\right)t\leq 0$ for every $k\in \mathbb{Z}$ and $t\geq 0$, we see that
\begin{equation}\label{LE6a}
\left\|S(t)\phi\right\|_{H^s} = (2\pi)^{\frac{1}{2}}\left\|\langle k \rangle^s e^{\eta(|k|-k^2)t}\widehat{\phi}(k)\right\|_{l^2(\mathbb{Z})}\leq \left\|\phi\right\|_{H^s}.
\end{equation} 
\eqref{LE6a} implies inequality \eqref{LE5}. To prove \eqref{LE6}, we assume that $s< 0$. Since $0 \leq T \leq 1$, we have
 $$t\leq \frac{(1+k^2t)}{(1+k^2)}, \text{ for all }k\in \mathbb{Z}, \ t\in [0,T].$$ 
So, it follows that
 \begin{equation}\label{LE7}
t^{|s|/2}\left\|S(t)\phi\right\|_{L^2} \leq  \left\|\left\langle t^{1/2}k\right\rangle^{|s|} e^{\eta\left(|k|-k^2\right)t}\right\|_{l^{\infty}(\mathbb{Z})} \left\|\phi\right\|_{H^s}.
\end{equation}
Using Lemma \ref{LLE0} we obtain
\begin{align}\label{LE8}
\left\langle  t^{1/2}k\right\rangle^{|s|} e^{\eta\left(|k|-k^2\right)t} & \lesssim_{s} 1+(tk^2)^{\frac{|s|}{2}}  e^{\eta\left(|k|-k^2\right)t} \\
& \lesssim_{s} 1+ \left(t^{\frac{|s|}{2}}+\eta^{-\frac{|s|}{2}}\right) e^{\frac{\eta}{8}\left(t+t^{\frac{1}{2}}\sqrt{t+\frac{8|s|}{\eta}}\right)}.
\end{align}
Therefore, we conclude \eqref{LE6} from \eqref{LE7} and \eqref{LE8}.
\end{proof}
\section{Bilinear estimate}

In this section, we establish the crucial bilinear estimates.

\begin{prop}\label{PROP2} Let $0\leq T\leq 1$ and $-\frac{1}{2}<s < 0$, then
\begin{equation}
\left\|\int_{0}^t S(t-t')\partial_x(uv)(t') \ dt'\right\|_{X_T^s} \lesssim_{s,\eta} T^{\frac{1+2s}{4}}\left\|u\right\|_{X_T^s}\left\|v\right\|_{X_T^s},
\end{equation}
for all $u,v\in X_T^s$
\end{prop}

\begin{proof}
Since $s< 0$, it follows that $(1+k^2)^{ \frac{s}{2}}\leq |k|^s$, for all integer $k$ different from zero. So, we deduce that
\begin{equation}\label{BE0}
\begin{aligned}
&  \left\|\int_{0}^t S(t-t')\partial_x(uv)(t') \ dt'\right\|_{H^s} \\
& \hspace{30pt} \leq (2\pi)^{1/2} \int_{0}^t \left\|\left\langle k \right\rangle^s e^{\eta\left(|k|-k^2\right)(t-t')} \left(\partial_x(uv)(t')\right)^{\wedge}(k) \right\|_{l^2(\mathbb{Z})} \ dt' \\
& \hspace{30pt} \leq (2\pi)^{1/2} \int_{0}^t \left\||k|^{1+s}e^{\eta\left(|k|-k^2\right)(t-t')}\right\|_{l^2(\mathbb{Z})}\left\| \widehat{u(t')}\ast \widehat{v(t')}(k) \right\|_{l^{\infty}(\mathbb{Z})} \ dt'. 
\end{aligned}
\end{equation}
The Young inequality implies that
\begin{equation}\label{BE1}
\left\| \widehat{u(t')}\ast\widehat{v(t')}(k) \right\|_{l^{\infty}(\mathbb{Z})}\leq \frac{1}{2\pi}\left( \frac{\left\|u\right\|_{X_T^s}\left\|v\right\|_{X_T^s}}{|t'|^{|s|}} \right),
\end{equation}
hence, we obtain
\begin{equation}\label{BE2}
\begin{aligned}
& \left\|\int_{0}^t S(t-t')\partial_x(uv)(t') \ dt'\right\|_{H^s}\\
&\hspace{30pt}\lesssim  \int_{0}^t  \frac{\left\||k|^{1+s}e^{\eta\left(|k|-k^2\right)( t-t')}\right\|_{l^2(\mathbb{Z})}}{|t'|^{|s|}}\ dt' \;\left\|u\right\|_{X_T^s}\left\|v\right\|_{X_T^s} .
\end{aligned}
\end{equation}
To estimate the integral on the right-hand side of \eqref{BE2}, we have from Lemma \ref{LLE1} that
\begin{equation}
\left\||k|^{s+1}e^{\eta\left(|k|-k^2\right)t}\right\|_{l^2(\mathbb{Z})}\lesssim_{s}\left(1+\frac{1}{(\eta t)^{\frac{s+1}{2}}}+\frac{1}{(\eta t)^{\frac{2s+3}{4}}}\right), \forall t>0.
\end{equation}
So, from \eqref{BE1}, \eqref{BE2} and taking $z=t'/t$ we get that 
\begin{equation}\label{BE3}
\begin{aligned}
& \left\|\int_{0}^t S(t-t')\partial_x(uv)(t') \ dt'\right\|_{H^s}  \\
& \quad \lesssim_{s}\left(\frac{t^{1+s}}{1+s}+\frac{t^{\frac{1+s}{2}}}{{\eta^{\frac{1+s}{2}}}}\int_{0}^1 z^s|1-z|^{-(\frac{1+s}{2})}dz+
\frac{t^{\frac{1+2s}{4}}}{{\eta^{\frac{3+2s}{4}}}}\int_{0}^1 z^s|1-z|^{-(\frac{3+2s}{4})}dz \right)\\
&\quad \quad \quad \cdot \left\|u\right\|_{X_T^s}\left\|v\right\|_{X_T^s} \\
& \quad \lesssim_{s}\left(1+\frac{1}{\eta^{\frac{1+s}{2}}}+\frac{1}{\eta^{\frac{3+2s}{4}}}\right)T^{\frac{1+2s}{4}}\left\|u\right\|_{X_T^s}\left\|v\right\|_{X_T^s}. 
\end{aligned}
\end{equation}
Arguing in a similar way as above, we have for all $0\leq t\leq T$ that
\begin{equation}\label{BE4}
\begin{aligned}
&t^{|s|/2}  \int_{0}^t \left\| S(t-t')\partial_x(uv)(t') \right\|_{L^2(\mathbb{T})}  \ dt' \\
& \quad \lesssim T^{|s|/2} \int_{0}^t  \frac{\left\||k|e^{\eta\left(|k|-k^2\right)(t-t')}\right\|_{l^2(\mathbb{Z})}}{|t'|^{|s|}}\ dt' \; \left\|u\right\|_{X_T^s}\left\|v\right\|_{X_T^s} \\
& \quad \lesssim_s T^{|s|/2} \int_{0}^t \left(\frac{1}{|t'|^{|s|}}+\frac{1}{(\eta (t- t'))^{1/2}|t'|^{|s|}} +\frac{1}{(\eta (t- t'))^{3/4}|t'|^{|s|}}\right) \,dt'\\
&\quad \quad \quad \cdot \left\|u\right\|_{X_T^s}\left\|v\right\|_{X_T^s} \\
&\quad \lesssim_{s} \left(1+\frac{1}{\eta^{1/2}} + \frac{1}{\eta^{3/4}}\right)T^{\frac{1+2s}{4}}\left\|u\right\|_{X_T^s}\left\|v\right\|_{X_T^s}. 
\end{aligned}
\end{equation}
This completes the proof.
\end{proof}

\begin{rem}
If we consider $s'>s>-\frac{1}{2}$. Then modifying the space $X_{T}^{s'}$ by
$$\tilde{X}_{T}^{s'}=\left\{u\in X_{T}^{s'}: \left\|u\right\|_{\tilde{X}_{T}^{s'}}<\infty \right\},$$
where 
$$ \left\|u\right\|_{\tilde{X}_{T}^{s'}}= \left\|u\right\|_{X_{T}^{s'}}+t^{|s|/2} \left\|(1-\partial_x^2)^{\frac{ s'-s}{2}} u\right\|_{L^2}$$
and using that
$$(1+k^2)^{s/2}\lesssim (1+k^2)^{s/2}(1+j^2)^{(s'-s)/2}+(1+k^2)^{s/2}\left(1+(k-j)^2\right)^{(s'-s)/2},$$
for all $k,j\in \mathbb{Z}$, we deduce arguing as in Proposition \eqref{PROP2} that
$$\left\|\int_{0}^t S(t-t')\partial_x(uv)(t') \ dt'\right\|_{\tilde{X}_T^{s'}} \lesssim_{s,\eta} T^{\frac{1+2s}{4}}\left(\left\|u\right\|_{\tilde{X}_T^{s'}}\left\|v\right\|_{X_T^s}+\left\|u\right\|_{X_T^s}\left\|v\right\|_{\tilde{X}_T^{s'}}\right).$$
\end{rem}

\begin{prop}\label{PROP3}
Let $0\leq T \leq 1$, $s \in (-\frac{1}{2},0)$ and $\delta\in [0,s+\frac{1}{2})$, then the application 
$$t\rightarrow \int_{0}^t S(t-t')\partial_x(u^2)(t')\ dt' ,$$
is in $C\left([0,T];H^{s+\delta}(\mathbb{T})\right)$, for every $u\in X_T^s$.
\end{prop}
\begin{proof}
Let $t,\tau \in [0,T]$ be fixed with $t<\tau$. Then, by Minkowski inequality, we see that 
\begin{equation}
\left\|\int_{0}^{\tau} S(\tau-t')\partial_x(u^2)(t')dt'-\int_{0}^t S(t-t')\partial_x(u^2)(t')dt'\right\|_{H^{s+\delta}} \leq \mathbb{I}(t,\tau)+\mathbb{II}(t,\tau),
\end{equation} 
where
\begin{align*}
 \mathbb{I}(t,\tau)&:=\int_{0}^t \left\|\left( S(\tau-t')-S(t-t')\right)\partial_x(u^2)(t')\right\|_{H^{s+\delta}}dt' \nonumber \\
\intertext{and}
\mathbb{II}(t,\tau)&:=\int_{t}^{\tau} \left\|S(\tau-t')\partial_x(u^2)(t')\right\|_{H^{s+\delta}}d t'. \nonumber
\end{align*}
Following the same ideas of the proof of Proposition \ref{PROP2} we obtain
\begin{align} \label{BE5}
 &\mathbb{II}(t,\tau)  \lesssim \int_{t}^{\tau} \left\|\left\langle k \right\rangle^{s+\delta}e^{\eta\left(|k|-k^2\right)(\tau-t')}\left(\partial_{x}u^2(t')\right)^{\wedge}(k)\right\|_{l^2(\mathbb{Z})} \ dt' \nonumber\\
& \lesssim_{s,\delta,\eta} \int_{t}^{\tau}\bigl((t')^s+(t')^s|\tau-t'|^{-\frac{1+s+\delta}{2}}+(t')^s|\tau-t'|^{-\frac{3+2( s+\delta)}{4}}\bigr) \ dt' \;\left\|u\right\|_{X_{T}^s}^2 \nonumber \\
& \lesssim_{s,\delta,\eta} \left(\frac{(\tau-t)^{1+s}}{1+s}+\int_{t}^{\tau} \bigl(|t'-t|^s|\tau-t'|^{-\frac{1+s+\delta}{2}}+|t'-t|^s|\tau-t'|^{-\frac{3+2( s+\delta)}{4}}\bigr) dt'\right) \nonumber \\
& \quad \quad \quad  \cdot \left\|u\right\|_{X_{T}^s}^2. 
\end{align}
But with the change of variable $z=\frac{t'-t}{\tau-t}$, we have that
\begin{align}\label{BE6}
&\int_{t}^{\tau}|t'-t|^s|\tau-t'|^{-\frac{1+s+\delta}{2}} \ dt' = (\tau-t)^{\frac{1+s-\delta}{2}}\int_{0}^{1}z^s|1-z|^{-\frac{1+s+\delta}{2}} \ dz,  \nonumber \\
&\int_{t}^{\tau}|t'-t|^s|\tau-t'|^{-\frac{3+2( s+\delta)}{4}} \ dt' =(\tau-t)^{\frac{1+2(s-\delta)}{4}} \int_{0}^1 z^s|1-z|^{-\frac{3+2( s+\delta)}{4}} \ dz.
\end{align}
Therefore, combining \eqref{BE5}, \eqref{BE6} and the hypothesis, we deduce that \\ $\lim_{\tau \to t}\mathbb{II}(t,\tau)=0$.
To estimate $\mathbb{I}(t,\tau)$, we observe that
\begin{equation}\label{BE7}
\mathbb{I}(t,\tau) \lesssim \int_{0}^{t} \frac{\left\||k|^{1+s+\delta}\left(e^{\left(iq(k)-p(k)\right) (\tau-t')}-e^{\left(iq(k)-p(k)\right)(t-t')}\right)\right\|_{l^2(\mathbb{Z})}}{|t'|^{|s|}} \, dt' \, \left\|u\right\|_{X_{T}^s}^2 .
\end{equation}
Applying Lemma \ref{LLE1}, for all $t'\in [0,t)$, we have that
\begin{align}\label{BE8}
&\left\||k|^{1+s+\delta}\left(e^{\left(iq(k)-p(k)\right)(\tau-t')}-e^{\left(iq(k)-p(k)\right)(t-t')}\right)\right\|_{l^2(\mathbb{Z})} \nonumber\\
&\hspace{30pt} \lesssim \left\||k|^{1+s+\delta}e^{\eta\left(|k|-k^2\right) (t-t')}\right\|_{l^2(\mathbb{Z})}  \nonumber \\
&\hspace{30pt} \lesssim_{s,\delta}\Upsilon_{\eta}^{1+s+\delta}(t-t'),
\end{align}
where $$\Upsilon_{\eta}^{1+s+\delta}(t)=1+\frac{1}{(\eta t)^{\frac{1+s+\delta}{2}}}+\frac{1}{(\eta t)^{\frac{3+2(s+\delta)}{4}}}, \quad \forall t>0.$$
Then, it follows from \eqref{BE8} and Weierstrass M-test that
\begin{equation}\label{BE9}
\lim_{\tau \to t} \ \left\||k|^{1+s+\delta}\left(e^{\left(iq(k)-p(k)\right)(\tau-t')}-e^{\left(iq(k)-p(k)\right)(t-t')}\right)\right\|_{l^2(\mathbb{Z})}=0,
\end{equation}
 for all $t'\in [0,t)$. Moreover, since $(t')^s \Upsilon_{\eta}^{1+s+\delta}(t-t')$ is in $L^1_{t'}(0,t)$, we deduce from \eqref{BE9} and Lebesgue dominated convergence theorem that $\lim_{\tau \to t}\mathbb{I}(t,\tau)=0$. This completes the proof.
\end{proof}

The next lemma is an adaptation of Lemma 2.3.1 in \cite{D} to the periodic case. This result allows us to adapt the above propositions to $C\left([0,T];H^s(\mathbb{T})\right)$, when $s\geq 0$ and $0<T\leq 1$. For the sake of completeness, we will sketch a proof.

\begin{lema}\label{LBE1}
Suppose $a>0$, $r\geq 0$ are real numbers and $\phi,\psi \in H^r(\mathbb{T})$. Then
\begin{equation}
\left\|\left\langle ak \right\rangle^r(\phi \psi)^{\wedge}(k)\right\|_{l^{\infty}(\mathbb{Z})}\leq 2^{\frac{r}{2}} \left\|\left\langle ak \right\rangle^{r}\widehat{\phi}(k)\right\|_{l^2(\mathbb{Z})}\left\|\left\langle ak \right\rangle^{r}\widehat{\psi}(k)\right\|_{l^2(\mathbb{Z})}.
\end{equation}
\end{lema}

\begin{proof}
Since $r\geq 0$, $H^r(\mathbb{T}) \hookrightarrow L^2(\mathbb{T})$. Therefore, $\phi,\psi \in L^2(\mathbb{T})$ and the exchange formula holds,   
$$(\phi\psi)^{\wedge}(k)=\widehat{\phi}\ast \widehat{\psi}(k)= \sum_{j=-\infty}^{\infty}\widehat{\phi}(j)\widehat{\psi}(k-j), \quad \forall k\in \mathbb{Z}.$$
To prove this lemma we will use Peetre's inequality which says that 
$$\left(1+|x|^2\right)^{\rho}\leq 2^{|\rho|}\left(1+|x-y|^2\right)^{|\rho|}\left(1+|y|^2\right)^{\rho}, \quad \forall x,y,\rho\in\mathbb{R}.$$
Thus,
$$\left\langle ak \right\rangle^r \leq 2^{\frac{r}{2}}\left\langle a(k-j)\right\rangle^r\left\langle aj \right\rangle^r, \quad \forall k \in \mathbb{Z}.$$
So, we have by Young's convolutions inequality that
\begin{align*}
\left|\left\langle ak \right\rangle^r(\phi \psi)^{\wedge}(k)\right| &\leq 2^{\frac{r}{2}} \sum_{j=\infty}^{\infty} \left|\left\langle aj \right\rangle^r \widehat{\phi}(j)\left\langle a(k-j)\right\rangle^r \widehat{\psi}(k-j)\right|\ \\
& \leq 2^{\frac{r}{2}}\left\|\left\langle ak \right\rangle^{r}\widehat{\phi}(k)\right\|_{l^2(\mathbb{Z})}\left\|\left\langle ak \right\rangle^{r}\widehat{\psi}(k)\right\|_{l^2(\mathbb{Z})}.
\end{align*}
\end{proof}

\begin{rem}\label{RBE1}
Assuming that $s\geq 0$ and $0<T\leq 1$, we have a similar result as the one obtained in Proposition \ref{PROP2} for the space $C\left([0,T];H^s(\mathbb{T})\right)$. In fact, we have that
$$\left\|\int_{0}^t S(t-t')\partial_x(uv)(t') \ dt'\right\|_{L^{\infty}_tH^s_x} \lesssim_{s,\eta} T^{\frac{1}{4}}\left\|u\right\|_{L^{\infty}_tH^s_x}\left\|v\right\|_{L^{\infty}_tH^s_x},$$
for all $u,v\in C\left([0,T];H^s(\mathbb{T})\right).$
To see this, we can use Lemma \ref{LBE1} with $a=1$ and Lemma \ref{LLE1}.
\begin{align} 
 \int_{0}^{t}  &\left\|S(t-t')\partial_x(uv)(t')\right\|_{H^s}dt'  \nonumber \\
 &\lesssim \int_{0}^t \left\| |k| e^{\eta\left(|k|-k^2\right)(t-t')} \right\|_{l^2(\mathbb{Z})} \left\|\left\langle k \right\rangle^s \left(uv(t')\right)^{\wedge}(k)\right\|_{l^{\infty}(\mathbb{Z})}  \ dt' \nonumber \\
 & \lesssim_s \int_{0}^{t} \left\||k| e^{\eta\left(|k|-k^2\right)(t-t')}\right\|_{l^2(\mathbb{Z})}\left\|u(t')\right\|_{H_x^s}\left\|v(t')\right\|_{H^s_x}\ dt'  \nonumber \\
 & \lesssim_{s,\eta} \int_{0}^{t} \Bigl(1+\frac{1}{(t-t')^{1/2}}+\frac{1}{(t-t')^{3/4}}\Bigr)\, dt' \;\left\|u\right\|_{L^{\infty}_tH^s_x}\left\|v\right\|_{L^{\infty}_tH^s_x} \nonumber \\
 & \lesssim_{s,\eta} T^{\frac{1}{4}}\left\|u\right\|_{L^{\infty}_tH^s_x}\left\|v\right\|_{L^{\infty}_tH^s_x}.
\end{align}
\end{rem}

\begin{rem}\label{RBE2}
Let $s\geq 0$ and $0<T\leq 1$. We have the same result given in Proposition \ref{PROP3}, changing $X_{T}^s$ by $C\left([0,T];H^s(\mathbb{T})\right)$ and taking $\delta\in[0,\frac{1}{2})$. In fact, considering $t,\tau\in [0,T]$ fixed with $t<\tau$, we define the terms $\mathbb{I}(t,\tau)$ and $\mathbb{II}(t,\tau)$ as in Proposition \ref{PROP3}. Then Lemma \ref{LLE1} implies that
\begin{align} 
 \mathbb{II}(t,\tau) & \lesssim \int_{t}^{\tau} \left\||k| \left\langle k \right\rangle^{\delta}e^{\eta\left(|k|-k^2\right)(\tau-t')}\right\|_{l^2(\mathbb{Z})}\left\|\left\langle k \right\rangle^{s}[u^2(t')]^{\wedge}(k)\right\|_{l^{\infty}(\mathbb{Z})}\ dt'  \nonumber \\
& \lesssim_{s,\eta} \int_{t}^{\tau} \left\||k|^{1+\delta}e^{\eta(|k|-k^2)(\tau-t')}\right\|_{l^2(\mathbb{Z})} \left\|u(t')\right\|_{H^s_x}^2\ dt'  \nonumber \\
& \lesssim_{s,\eta,\delta} \int_{t}^{\tau} \Bigl(1+\frac{1}{(\tau-t')^{\frac{1+\delta}{2}}}+\frac{1}{(\tau-t')^{\frac{3+2\delta}{4}}}\Bigr) \, dt' \;\left\|u\right\|_{L^{\infty}_tH^s_x}^2 \nonumber \\ 
& \lesssim_{s,\eta,\delta} \left((\tau-t)+(\tau-t)^{\frac{1-\delta}{2}}+(\tau-t)^{\frac{1}{4}(1-2\delta)} \right)\left\|u \right\|_{L^{\infty}_tH^s_x}^2,
\end{align}
then is clear that $\lim_{\tau \to t} \mathbb{II}(t,\tau)=0$. Also, we observe that,
\begin{equation}
\mathbb{I}(t,\tau) \lesssim_{s,\delta}\int_{0}^{t} \nor{|k|^{1+\delta}\Bigl(e^{\left(iq(k)-p(k)\right)(\tau-t')}-e^{\left(iq(k)-p(k)\right)(t-t')}\Bigr)}{l^2(\mathbb{Z})}dt' \nora{u}{L^{\infty}_tH^s_x}{2},
\end{equation}
and again Lemma \ref{LLE1} implies that
\begin{align*}
&\left\||k|^{1+\delta}\left(e^{\left(iq(k)-p(k)\right)(\tau-t')}-e^{\left(iq(k)-p(k)\right)(t-t')}\right)\right\|_{l^2(\mathbb{Z})} \\ 
&\qquad \lesssim_{s,\delta}\left(1+\frac{1}{\left(\eta(t-t')\right)^{\frac{1+\delta}{2}}}+\frac{1}{\left(\eta(t-t')\right)^{\frac{3+2\delta}{4}}}\right), \ \forall t>0.
\end{align*}
Therefore, using the above inequalities, we can argue as in the proof of Proposition \ref{PROP3} and conclude that $\lim_{\tau \to t} \mathbb{I}(t,\tau)=0$.
\end{rem}

\section{Well-Posedness}

In this section we show that the Cauchy problem \eqref{cl} is locally and globally well-posed in $H^s(\mathbb{T})$ for $s > -\frac{1}{2}$. In fact, to prove the local existence result, we will construct a contraction with the integral formulation  \eqref{intequation}. The principal argument to obtain our desired results is to use the bilinear estimates in Proposition \ref{PROP2} and Remark \ref{RBE1} for the nonlinear part $\partial_x(u^2)$ in the convenient $\mathfrak{X}_T^s$ spaces. We begin this section giving a proof of Theorem \ref{mainresult}

\begin{proof}[Proof of Theorem \ref{mainresult}] For $T\in(0,1]$, we consider the space $\mathfrak{X}_T^s=X_T^s$, if $-\frac{1}{2}< s < 0$, and if $s\geq 0$, we take $\mathfrak{X}_T^s=C\left([0,T];H^s(\mathbb{T})\right)$. We divide the proof in four steps.\\

\emph{1. Existence}. Let $\phi \in H^s(\mathbb{T})$, $s> -\frac{1}{2}$. We define the application
$$\Psi(u)=S(t)\phi-\frac{1}{2}\int_{0}^t S(t-t')\partial_x(u^2(t')) \ dt', \text{ for each } u \in \mathfrak{X}_T^s.$$
By Proposition \ref{PROP1}, together with Proposition \ref{PROP2} when $s< 0$, or Remark \ref{RBE1} when $s\geq 0$, there exists a positive constant $C=C(\eta,s)$, independent of $\beta$, such that for all $u, v \in \mathfrak{X}_T^s$ and $0<T\leq 1$
\begin{align}
\left\|\Psi(u)\right\|_{\mathfrak{X}_T^s} &\leq C\left(\left\|\phi\right\|_s+T^{g(s)}\left\|u\right\|_{\mathfrak{X}_T^s}^2\right), \label{WP1} \\
\left\|\Psi(u)-\Psi(v)\right\|_{\mathfrak{X}_T^s} & \leq C T^{g(s)}\left\|u-v\right\|_{\mathfrak{X}_T^s}\left\|u+v\right\|_{\mathfrak{X}_T^s}, \label{WP2}
\end{align}
where $g(s)=\frac{1}{4}(1+2s)$, for all $s\in (-\frac{1}{2},0)$, and $g(s)=\frac{1}{4}$, if $s\geq 0$. Then, let $E_{T}(\gamma)=\left\{u\in \mathfrak{X}_T^s : \left\|u \right\|_{\mathfrak{X}_T^s} \leq \gamma \right\}$, where $\gamma=2C\left\|\phi \right\|_s$ and \\ $0<T \leq \min \left\{1,\left(4C\gamma\right)^{-\frac{1}{g(s)}} \right\}$. The estimates \eqref{WP1} and \eqref{WP2} imply that $\Psi$ is a contraction on the complete metric space $E_T(\gamma)$. Therefore, we deduce by the Fixed Point Theorem that exists an unique solution $u$ of the integral equation \eqref{intequation} in $E_T(\gamma)$ and with initial data $u(0)=\phi$. \\

\emph{2. Continuous dependence}. Let $\phi_1,\phi_2 \in H^s(\mathbb{T})$ and $u_1\in \mathfrak{X}_{T_1}^s$, $u_2\in \mathfrak{X}_{T_2}^s$ be the respective solutions of the Chen-Lee equation constructed in the subsection of \emph{Existence} above. We recall that the solutions and the times of existence satisfy
\begin{align*}
&\left\|u_i\right\|_{\mathfrak{X}_T^s} \leq 2C\left\|\phi_i\right\|_{H^s}, \\
&0<T_i\leq \min\left\{1, \left(8C^2\left\|\phi_i\right\|_{H^s}\right)^{-\frac{1}{g(s)}}\right\},
\end{align*}
for $i=1,2$, and $C=C(\eta,s)$. Therefore, by Proposition \ref{PROP1}, together with Proposition \ref{PROP2} when $s<0$, or by Remark \ref{RBE1} when $s\geq 0$, we have that for all $T\in \left(0, \min\left\{T_1,T_2\right\}\right]$ 

\begin{align*}
\left\|u_1-u_2\right\|_{\mathfrak{X}_T^s} & \leq C\left\|\phi_1-\phi_2\right\|_{H^s}+ C T^{g(s)}\left\|u_1+u_2\right\|_{\mathfrak{X}_T^s}\left\|u_1-u_2\right\|_{\mathfrak{X}_T^s}  \\
 & \leq C\left\|\phi_1-\phi_2\right\|_{H^s}+ \frac{\left(\left\|\phi_1\right\|_{H^s}+\left\|\phi_2\right\|_{H^s}\right)}{4\max_{i=1,2}\left\{\left\|\phi_i\right\|_{H^s}\right\}}\left\|u_1-u_2\right\|_{\mathfrak{X}_T^s}  \\
&\leq C\left\|\phi_1-\phi_2\right\|_{H^s}+ \frac{1}{2}\left\|u_1-u_2\right\|_{\mathfrak{X}_T^s},
\end{align*}
but this implies that 
$$\left\|u_1(t)-u_2(t)\right\|_{H^s}\leq \left\|u_1-u_2\right\|_{\mathfrak{X}_T^s}\leq 2C\left\|\phi_1-\phi_2\right\|_{H^s}, \, \text{for all } t \in [0,T].$$

\emph{3. Uniqueness}. We shall proof the uniqueness of solutions to the integral equation \eqref{intequation} in the space $\mathfrak{X}_T^s$, where $T$ is defined as in the subsection of \emph{Existence}. Let $u,v \in \mathfrak{X}_T^s$ be solutions of the integral equation \eqref{intequation}  on the time interval $[0,T]$ with the same initial data $\phi$. Arguing as in the proof of Proposition \ref{PROP2} for $-\frac{1}{2}<s< 0$ or as in the Remark \ref{RBE1} when $s\geq 0$, there exists $C=C(\eta,s)$ such that for all $0<T_1\leq T$
\begin{equation}\label{WP3}
\left\|u-v\right\|_{\mathfrak{X}_{T_1}^s}\leq C K T_1^{g(s)}\left\|u-v\right\|_{\mathfrak{X}_{T_1}^s},
\end{equation}
where $K:=\left\|u\right\|_{\mathfrak{X}_{T}^s}+\left\|v\right\|_{\mathfrak{X}_{T}^s}$. Taking $T_1 \in \bigl(0,(CK)^{-\frac{1}{g(s)}}\bigr)$, we deduce from \eqref{WP3} that $u \equiv v$ on $[0,T_1]$. Thus, iterating this argument, we extend the uniqueness result to the whole interval $[0,T]$.\\

 \emph{4. The solution $ u\in C\left((0,T],H^{\infty}(\mathbb{T})\right)$ }. Using Lemma \ref{LLE0} and arguing as in the proof of Proposition 2.2 in \cite{BI}, we have that the map $t\mapsto S(t)\phi$ is continuous in the interval $(0,T]$ with respect to the topology of $H^{\infty}(\mathbb{T})$. Since our solution $u$ is in $\mathfrak{X}_T^s$, we deduce from Proposition \ref{PROP3} or Remark \ref{RBE2}, that there exists $\lambda>0$, such that 
$$u\in C\left([0,T];H^s(\mathbb{T})\right)\cap C\left((0,T];H^{s+\lambda}(\mathbb{T})\right).$$ 
Therefore we can iterate this argument, using the uniqueness result and the fact that the time of existence of solutions depends uniquely on the $H^s(\mathbb{T})$-norm of the initial data. Thus we deduce that 
$$u\in C\left([0,T];H^s(\mathbb{T})\right)\cap C\left((0,T];H^{\infty}(\mathbb{T})\right).$$

\end{proof}

Next, we will give a proof of Theorem \ref{globalresult}

\begin{proof}[Proof of Theorem \ref{globalresult}]
We divide the proof in two steps.\\

\emph{1.} Let $s\geq 0$ and $T^*(\left\|\phi\right\|_s)$ defined as
$$T^*=\sup\left\{T>0: \exists !\text{ solution of \eqref{intequation} in }  C\left([0,T];H^s(\mathbb{T})\right)\right\}.$$
Let $u \in C\left([0,T^*);H^s(\mathbb{T})\right)\cap C\left((0,T^*);H^{\infty}(\mathbb{T})\right)$ be the local solution of the integral equation \eqref{intequation} and defined in the maximal time interval $[0,T^*)$. We will prove that $T^*<\infty$ implies a contradiction. Since $u$ is smooth by Theorem \ref{mainresult}, we have that $u$ solves \eqref{intequation} in a classic sense, therefore we can take the $L^2(\mathbb{T})$ inner product between $u$ and \eqref{intequation} to obtain 
\begin{align*}
\frac{1}{2}\frac{d}{dt}\left\|u(t)\right\|_0^2&=(u,u_t)_0 \\
&=-(u,uu_x)_0-\beta(u,\mathcal{H}u_{xx})_0-\eta(u,\mathcal{H}u_x)_0+\eta(u,u_{xx})_0 \\
&=2\pi\eta\sum_{k=-\infty}^{\infty}\left(|k|-k^2\right)|\Hat{u}(t)|^2 \\
&\leq 0,
\end{align*}
where we have used that $(|k|-k^2)\leq 0$ for all integer $k$. Thus we obtain
\begin{equation}
\left\|u(t)\right\|_0 \leq \left\|\phi\right\|_0 \, \text{for all } t \in [0,T^*).
\end{equation}
Since the time of existence $T(\cdot)$ is a decreasing function of the norm of the initial data, there exists  a time $\widetilde{T}>0$, such that for all $\psi \in H^1(\mathbb{T})$ with $\left\|\psi \right\|_{0}\leq \left\|\phi\right\|_{0}$, there exists a function $\tilde{u}\in C\left([0,\widetilde{T}];H^s(\mathbb{T})\right)$ solution of \eqref{intequation} with $\widetilde{u}(0)=\psi$. Let $0<\epsilon <\widetilde{T}$, applying the above result to $\psi=u(T^*-\epsilon)$, we define
\begin{equation}
v(t)=\begin{cases} 
      u(t), & \text{ when } \, 0 \leq t \leq T^*-\epsilon, \\
     \tilde{u}(t-T^*+\epsilon), &\text{ when } \, T^*-\epsilon \leq t \leq T^*+\widetilde{T}-\epsilon \\
   \end{cases}
\end{equation}
Then $v(t)$ is a solution of the integral equation \eqref{intequation} in $[0,T^*+\widetilde{T}-\epsilon]$, but this contradicts the definition of $T^*$, since $T^*+\widetilde{T}-\epsilon >T^*$. We have concluded the global result when $s\geq 0$.

\emph{2.} Let $s\in (-1/2,0)$, $\phi\in H^s(\mathbb{T})$ and $u\in X_{T}^s$ be the solution of the Cauchy problem \eqref{cl} given in Theorem \ref{mainresult}. Let $T'\in (0,T)$ fixed, we have that
$$
\left\|u\right\|_{X_{T'}^s}=M_{T',s}<\infty.
$$
Since $u\in C\left((0,T];H^{\infty}(\mathbb{T})\right)$, it follows that $u(T')\in L^2(\mathbb{T})$. Thus, the part $(i)$ of Theorem \ref{globalresult} implies that $\tilde{u}$ the solution of the integral equation \ref{intequation} with initial data $u(T')$ is global in time. Moreover, uniqueness implies that $\tilde{u}(t)=u(T'+t)$ for all $t\in [0,T-T']$. Therefore, we deduce that
\begin{align*}
\left\|u\right\|_{X_{T}^s} & \leq \left\|u\right\|_{X_{T'}^s}+\left\|u(T'+\cdot)\right\|_{X_{T-T'}^s} \\
&\leq M_{T',s}+\left\|\tilde{u}\right\|_{X_{T-T'}^s} \\
&= M_{T',s}+ \sup_{t\in [0,T-T']} \left\{\left\|\tilde{u}(t)\right\|_{s}+t^{|s|/2}\left\|\tilde{u}(t)\right\|_{L^2(\mathbb{T})}\right\} \\
& \leq M_{T',s}+ \left(1+(T-T') ^{|s|/2}\right)\sup_{t\in [0,T-T']}\left\|\tilde{u}(t)\right\|_{L^2(\mathbb{T})}. 
\end{align*}
The global result follows from the above estimate when $s\in (-1/2,0)$.
\end{proof}

\section{ Ill-posedness result}

From the Theorem \ref{mainresult}, it is known that $CL$ is locally well-posed for data $\phi \in H^s(\mathbb{T})$, $s>-1/2$. In fact the map data-solution turns out to be smooth. In this section we will prove that one cannot solve Cauchy problem \eqref{cl} applying Picard iterative method on the integral equation \eqref{intequation}, at least in the Sobolev spaces $H^s(\mathbb{T})$, with $s<-1$.  We first prove the next theorem.

\begin{teorem}\label{TIP1}
Let $s<-1$, $\beta,\eta>0$ and $T>0$. Then there does not exist a space $B_T^s$ continuously embedded in $C\left([0,T],H^s(\mathbb{T})\right)$,i.e.
$$\left\|u\right\|_{L_{t}^{\infty}H^s}\lesssim \left\|u\right\|_{B_T^s}, \, \forall u\in B_T^s$$
and such that
\begin{equation}\label{IP1}
\left\|S(t)\phi\right\|_{B_T^s}\lesssim \left\|\phi \right\|_{H^s}, \qquad \forall \phi \in H^s(\mathbb{T}),
\end{equation}
and 
\begin{equation}\label{IP2}
\left\|\int_0^t S(t-t')[u(t')u_x(t')]\ dt'\right\|_{B_T^s}\lesssim \left\|u \right\|_{B_T^s}^2, \, \forall u\in B_T^s.
\end{equation}
\end{teorem}

\begin{proof}
Let $s<-1$, $\beta,\eta>0$ and $T>0$. Suppose that there exists a space $B_T^s$, which satisfies the conditions given in the theorem. Take $\phi \in H^s(\mathbb{T})$ and $u(t)=S(t)\phi$, then \eqref{IP2} implies that
\begin{equation}
\left\|\int_0^t S(t-t')[(S(t')\phi)(S(t')\phi_x)]\ dt'\right\|_{B_T^s}\lesssim \left\|S(t)\phi \right\|_{B_T^s}^2.
\end{equation}
Since $B_T^s$ is densely embedded in $C\left([0,T],H^s(\mathbb{T})\right)$, we obtain using \eqref{IP1} that for each $t\in [0,T]$
\begin{equation}\label{IP3}
\left\|\int_0^t S(t-t')[(S(t')\phi)(S(t')\phi_x)]\ dt'\right\|_{H^s}\lesssim \left\|\phi \right\|_{H^s}^2.
\end{equation}
We will show that \eqref{IP3} fails for an appropriated function $\phi$. Take $\phi$ defined by its Fourier transform as
\[
 \widehat{\phi}(k)=
  \begin{cases}
   N^{-s} & \text{if } k=N \text{ or } k=1-N, \\
   0       & \text{otherwise }
  \end{cases}
\]
where $N>1$ is a positive integer. It is easy to see that $\left\|\phi\right\|_s^2 \sim_s 1$. From the definition of the group $\left(S(t)\right)_{t\geq 0}$, we have that
\begin{align*}
\int_0^t S(t-t')&[(S(t')\phi)(S(t')\phi_x)]\ dt' \\
&= \int_0^t \sum_{k}e^{\left(iq(k)-p(k)\right)(t-t')}e^{ikx}(ik)\left(\widehat{S(t')\phi}\ast \widehat{S(t')\phi}\right)(k) \ dt'\\
&= \int_0^t \sum_{k}e^{\left(iq(k)-p(k)\right)(t-t')}e^{ikx}(ik) \\
&\quad \quad \quad \cdot\Bigl(\sum_{j}e^{\left(iq(j)-p(j)\right)t'}e^{\left(iq(k-j)-p(k-j)\right)t'}\widehat{\phi}(j)\widehat{\phi}(k-j)\Bigr) \ dt' \\
&= \sum_{k,j} e^{\left(iq(k)-p(k)\right)t+ikx}(ik)\widehat{\phi}(j)\widehat{\phi}(k-j)\int_0^t e^{t'\left[i\psi(k,j)-\sigma(k,j)\right]} dt' 
\end{align*}
where 
$$\psi(k,j)=\beta\left[ (k-j)|k-j|-|k|k+|j|j \right]$$
and
$$\sigma(k,j)=\eta\left[(k-j)^2-|k-j|-|k|^2+|k|+|j|^2-|j| \right].$$
Thus the above argument and the definition of $\phi$ imply 
\begin{align*}
&\left( \int_0^t S(t-t')[(S(t')\phi)(S(t')\phi_x)]\ dt' \right)^{\wedge}(1) \\ 
&\hspace{30pt}= e^{\left(iq(1)-p(1)\right)t}(i)N^{-2s}\int_0^t e^{t'\left[i\psi(1,N)-\sigma(1,N)\right]} dt'+\\
&\hspace{30pt}\quad \quad \quad+e^{\left(iq(1)-p(1)\right)t}(i)N^{-2s}\int_0^t e^{t'\left[i\psi(1,1-N)-\sigma(1,1-N)\right]} dt' \\
&\hspace{30pt}= 2  e^{\left(iq(1)-p(1)\right)t}(i) N^{-2s}\int_0^t e^{t'\left[i\psi(1,N)-\sigma(1,N)\right]} dt'.
\end{align*}
Here we have used that $\psi(1,N)=\psi(1,1-N)=2\beta(N-1)$ and $\sigma(1,N)=\sigma(1,1-N)=2\eta (N^2-2N+1)$. Hence, it follows that
\begin{align}
&\left\|\int_0^t S(t-t')[(S(t')\phi)(S(t')\phi_x)]\ dt'\right\|_{H^s}^2 \nonumber \\
&\hspace{30pt}\gtrsim_s \left|\left( \int_0^t S(t-t')[(S(t')\phi)(S(t')\phi_x)]\ dt' \right)^{\wedge}(1)\right|^2 \nonumber \\
&\hspace{30pt} \sim_s \left| e^{\left(iq(1)-p(1)\right)t}N^{-2s}\frac{e^{t\left[i\psi(1,N)-\sigma(1,N)\right]}-1}{i\psi(1,N)-\sigma(1,N)}\right|^2 \nonumber \\ 
&\hspace{30pt} \gtrsim_s \left| N^{-2s}\Re\left(\frac{e^{t\left[i\psi(1,N)-\sigma(1,N)\right]}-1}{i\psi(1,N)-\sigma(1,N)}\right)\right|^2. \label{IP5}
\end{align}
Since $\sigma(1,N)\sim \eta N^2$ and $\sigma(1,N)\sim \beta N$ we deduce that
\begin{align*}
\sigma(1,N)\left(1- e^{-t\sigma(1,N)}\cos\left(t\psi(1,N)\right)\right) &\gtrsim \eta N^2\left(1- e^{-t\eta N^2}\right), \\
\psi(1,N)e^{-t\sigma(1,N)}\sin\left(t\psi(1,N)\right) &\gtrsim -\beta N e^{-t\eta N^2}
\end{align*}
and 
$$|\sigma(1,N)|^2+|\psi(1,N)|^2 \sim N^2\left(\eta^2N^2+\beta^2\right),$$
hence
\begin{equation}\label{IP6}
\Re\left(\frac{e^{t\left[i\psi(1,N)-\sigma(1,N)\right]}-1}{i\psi(1,N)-\sigma(1,N)}\right) \gtrsim \frac{\left(\eta-(\eta+\beta)e^{-t\eta N^2}\right)}{(\eta^2+\beta^2)N^2}
\end{equation}

Therefore, from \eqref{IP5} and \eqref{IP6} 
\begin{equation} 
\left\|\int_0^t S(t-t')[(S(t')\phi)(S(t')\phi_x)]\ dt'\right\|_{H^s} \gtrsim_s \frac{\left(\eta-(\eta+\beta)e^{-t\eta N^2}\right)}{(\eta^2+\beta^2)}N^{-2(s+1)},
\end{equation}
but this contradicts \eqref{IP3} for $N$ large enough, since $\left\|\phi\right\|_s^2\sim_s 1$ and $s<-1$.
\end{proof}
As a consequence of Theorem \ref{TIP1} we can obtain the Theorem \ref{malpuestodos}.
\begin{proof}[Proof of Theorem \ref{malpuestodos}] 
Let $s<-1$, suppose that there exists $T>0$ such that the Cauchy problem \eqref{cl} is locally well-posed in $H^s(\mathbb{T})$ on the time interval $[0,T]$ and such that the flow map $\Phi:H^s(\mathbb{T})\rightarrow C\left([0,T];H^s(\mathbb{T})\right)$ is $C^2$ at the origin. When $\phi \in H^s(\mathbb{T})$, we will denote as $u_{\phi}(t)=\Phi(t)\phi$ the solution of the Cauchy problem \eqref{cl} with initial data $\phi$. This means that $u_{\phi}$ is a solution of the integral equation 

$$u_{\phi}(t)=S(t)\phi-\frac{1}{2}\int_{0}^t S(t-t')\partial_x(u_{\phi})^2(t')dt'.$$

By computing the Fr\'echet derivative of $\Phi(t)$ at $\phi$ with direction $\psi$, we obtain
\begin{equation}\label{IP7}
d_{\phi}\Phi(t)(\psi)=S(t)\psi-\int_{0}^t S(t-t')\partial_x\left(u_{\phi}(t') d_{\phi}\Phi(t')(\psi)\right)dt'.
\end{equation}
Since the Cauchy problem \eqref{cl} is supposed to be well-posed, we know using the uniqueness that $\Phi(t)(0)=0$, so we deduce from \eqref{IP7} that
\begin{equation}\label{IP8}
d_{0}\Phi(t)(\psi)=S(t)\psi.
\end{equation}
Using \eqref{IP7} and \eqref{IP8} we can compute the second Fr\'echet derivative at the origin in the direction $(\phi,\psi)$
$$d^2_0\Phi(t)(\phi,\psi)=-\int_0^t S(t-t')\partial_x[(S(t')\phi)(S(t')\psi)]\ dt'.$$
Assumption of $C^2$ regularity implies that $d^2_0\Phi(t)\in \mathcal{B}\left(H^s(\mathbb{T})\times H^s(\mathbb{T}),H^s(\mathbb{T})\right)$, which would lead to the following inequality
\begin{equation}\label{IP9}
\left\|d^2_0\Phi(t)(\phi,\psi)\right\|_{H^s}\lesssim \left\|\phi\right\|_{H^s}\left\|\psi\right\|_{H^s}, \forall \phi,\psi \in H^s(\mathbb{T}).
\end{equation}
But inequality \eqref{IP9} is equivalent to \eqref{IP3}, which has been shown to fail in the proof of Theorem \ref{TIP1}.
\end{proof}

\subsection*{Acknowledgements}
The authors are supported by the Universidad Nacional de Colombia, sede Bogot\'a. The authors would like to thank the seminar on Evolution PDE for postgraduate students at the Universidad Nacional de Colombia, sede Bogot\'a, for all their help and comments.

\bibliographystyle{amsplain}

\begin{thebibliography}{30}


\bibitem[1]{BI} H. A. Biagioni, J.L. Bona, R. I\'orio and M. Scialom \textit{On the Korteweg-de Vries-Kuramoto-Sivashinsky equation}, Adv. Diff. Eq. \textbf{1} (1996), pp. 1-20.
\bibitem[2]{Bourgain} J. Bourgain,\textit{ Fourier transform restriction phenomena for certain lattice subsets and applications to nonlinear evolutons equations. II. The KdV equations}, Geom. Funct. Anal. \textbf{3} (1993), pp. 209-262.
%
%
\bibitem[3]{D} D. B. Dix, \textit{Temporal asymptotic behavior of solutions of the benjamin-ono-burgers equation}, J. Diff. Eq., \textbf{97}(1991), pp. 238-287.
%
\bibitem[4]{Dix} D. B. Dix, \textit{Nonuniqueness and uniqueness in the initial value problem for burgers' equation}, SIAM J. Math. Anal., \textbf{1} No. 1 (1996), pp. 1-17.
%
\bibitem[5]{Duque} O. Duque, \textit{Sobre una versi\'on bidimensional de la ecuaci\'on Benjamin-Ono generalizada}, PhD Thesis Universidad Nacional de Colombia (2014).
%
\bibitem[6]{Amin} S. A. Esfahani, \textit{High Dimensional Nonlinear Dispersive Models}, PhD Thesis IMPA (2008).
%
\bibitem[7]{FeKa} B. -F. Feng, T. Kawahara, \textit{Temporal evolutions and stationary waves for dissipative Benjamin-Ono equation}, Phys. D \textbf{139}(2000), pp. 301-318.
%
\bibitem[8]{KPV} C. E. Kenig, G. Ponce, L. Vega, \textit{ A bilinear estimate with applications to the KdV equation}, J. Amer.Math.Soc. 9 \textbf{2} (1996), pp. 573-603.
%
\bibitem[9]{CL} Y. C. Lee, H. H. Chen, \textit{Nonlinear dynamical models of plasma turbulence}, Phys. Scr. \textbf{T2/1}  (1982), pp. 41-47. 
%
\bibitem[10]{MolSauTzv} L. Molinet, J. C. Saut, and N. Tzvetkov, \textit{Ill-posedness issues for the Benjamin-Ono and related equations}, SIAM J. Math. Anal. \textbf{33} No. 4 (2001) pp. 982-988.
%
\bibitem[11]{P} R. Pastr\'an, \textit{On a Perturbation of the Benjamin-Ono Equation}, Nonlinear Anal. \textbf{93} (2013), pp. 273-296.
%
%
\bibitem[12]{KDV} D. Pilod, \textit{Sharp well-posedness results for the Kuramoto-Velarde equation}, Commun. Pure Appl. Anal. \textbf{7} No. 4 (2008), pp. 867-881.
%
%
\bibitem[13]{clq} S. Qian, Y. C. Lee, H. H. Chen, \textit{A study of nonlinear dynamical models of plasma turbulence}, Phys. Fluids B 1 \textbf{1}(1989), pp. 87-98.
%
\bibitem[14]{clq1} S. Qian, H. H. Chen, Y. C. Lee, \textit{A turbulence model with stochastic soliton motion}, J. Math. Phys. \textbf{31}(1990), pp. 506-516.
%
%
\end{thebibliography}

\end{document}